\documentclass[11pt]{article}

\usepackage{graphicx, amssymb, amsmath, amsthm, amsfonts, amsbsy, color, textcomp}
\usepackage[it,hang]{caption}
\usepackage{verbatim}
 
\usepackage{pdfsync}

\usepackage{hyperref}

\newtheorem{theorem}{Theorem}
\newtheorem{proposition}[theorem]{Proposition}
\newtheorem{corollary}[theorem]{Corollary}
\newtheorem{definition}{Definition}

\newcommand{\bi}{\begin{itemize}}
\newcommand{\ei}{\end{itemize}}
\newcommand{\be}{\begin{equation}}
\newcommand{\ee}{\end{equation}}
\newcommand{\ben}{\begin{equation*}}
\newcommand{\een}{\end{equation*}}
\newcommand{\bea}{\begin{eqnarray}}
\newcommand{\eea}{\end{eqnarray}}
\newcommand{\bean}{\begin{eqnarray*}}
\newcommand{\eean}{\end{eqnarray*}}

\newcommand{\bm}{\left ( \begin{matrix}}
\newcommand{\fm}{\end{matrix} \right ) }

\newcommand{\p}{\rho}

\usepackage{pxfonts}

\begin{document}
%
%

\title{Bi-banded Paths, a Bijection and \\the Narayana Numbers}
\author{Judy-anne Osborn\\
\, \\
Mathematical Sciences Institute,\\
Australian National University,\\
Canberra, ACT 0200, Australia\\
}
\date{}

\maketitle

\begin{abstract}
We find a bijection between bi-banded paths and peak-counting paths, applying to two classes of lattice paths including Dyck paths.  Thus we find a new interpretation of Narayana numbers as coefficients of weight polynomials enumerating bi-banded Dyck paths, which class of paths has arisen naturally in previous literature in a solution of the stationary state of the `TASEP' stochastic process.  
\end{abstract}

\section{Introduction} 

In a paper of Brak and Essam \cite{BE} the class of bi-banded Dyck paths arose in the solution of TASEP (Totally ASymmetric Exclusion Process) - a model for  particle-hopping with excluded volume \cite{Bl}. It was observed that the coefficients of the weight polynomial, or partition function, for bi-banded Dyck paths were the \emph{Narayana numbers}, which are known for enumerating Dyck paths by \emph{peaks} \cite{N1955,N1979}, but the authors did not obtain a bijective proof.  

In this paper we give a mapping between bi-banded Dyck paths and Dyck paths enumerated by peaks which provides a bijective proof of the equality observed in \cite{BE}.  Furthermore, we show that the same bijection applies to the broader class of bilateral Dyck paths which are not restricted to stay above the baseline.  To do this we use a representation of Dyck and bilateral Dyck paths which we call a \emph{checkmark} representation.

Bi-banded paths are a class of weighted lattice paths with edges weighted by alternating horizontal bands.  We give a formal definition of bi-banded paths in Section~\ref{sec:definitions}.   An example of the bijection which we will prove in Section~\ref{sec:bijection} is illustrated in the case of Dyck paths in Figure~\ref{fig:banded3by3bijectDyckThicklines}.

\begin{figure}[htbp]
\begin{center}
\includegraphics[width=120mm]{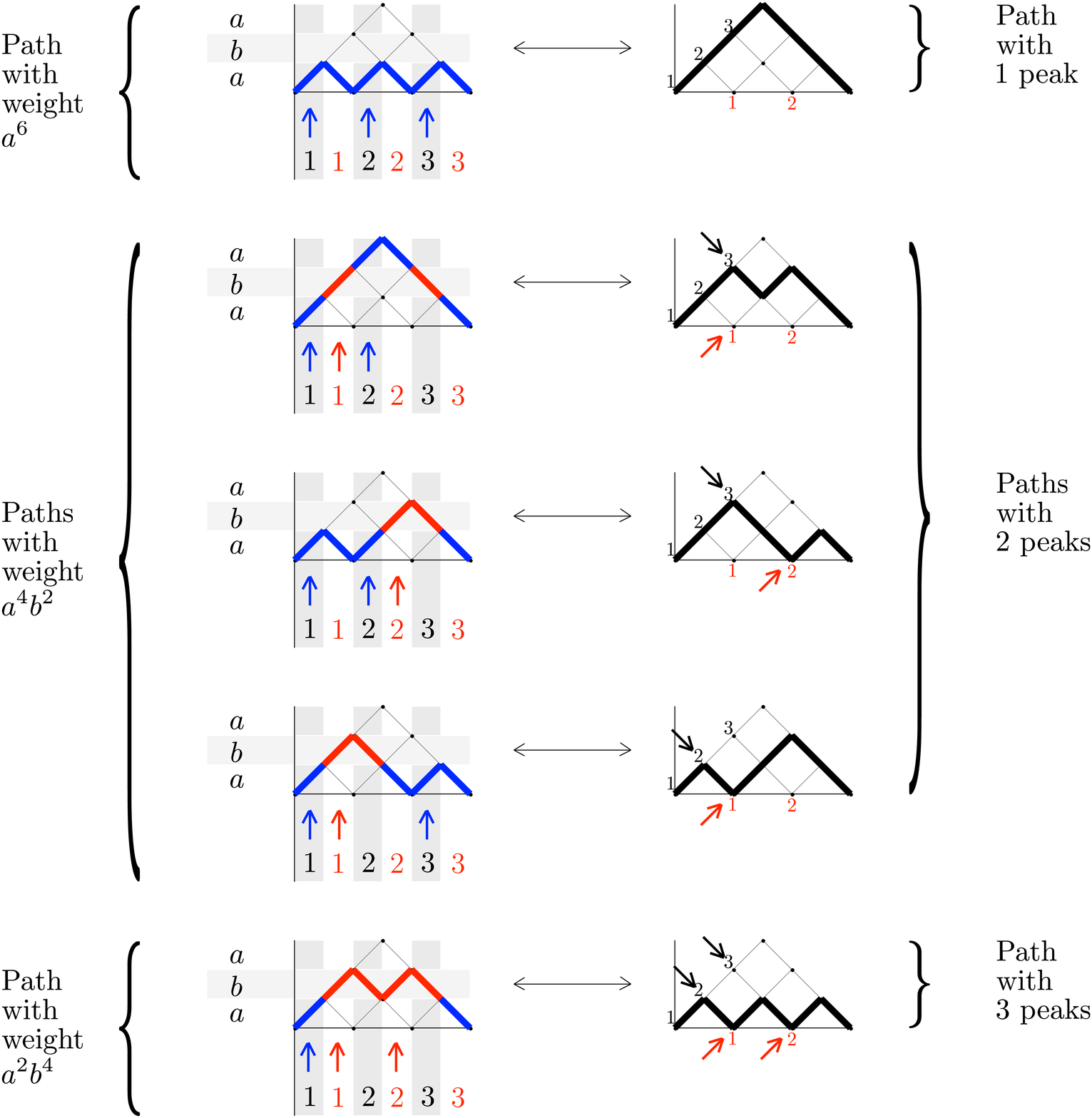}
\caption{An example of the bijection between bi-banded Dyck paths and peak-counting Dyck paths, for all Dyck paths of length $2n=6$.  The weight polynomials  are $W_6(a,b)=a^6+3a^4b^2+a^2b^4$ for bi-banded paths and $\widehat{W}_6(m)=m^3+3m^2+m$ for peak-counting paths.} 
\label{fig:banded3by3bijectDyckThicklines}
\end{center}
\end{figure}

\section{Definitions} \label{sec:definitions}

Within this paper, a \emph{path of length $t$} is an alternating sequence $p=v_0e_1v_1\hdots e_{t}v_{t}$ of vertices $\{v_i\}_{i=0,...,t}$ and edges $\{e_j\}_{j=1,...,t}$ constrained as follows.  Each edge is a pair a vertices $e_j:=(v_{j-1},v_j)$, classified as either an \emph{up-step} or a \emph{down-step}, accordingly as the difference $v_j-v_{j-1}$ equals $(1,1)$ or $(1,-1)$ respectively.  The \emph{initial vertex} $v_0=(0,0)$, the \emph{final vertex} $v_t=(t,0)$, and all vertices belong to a specified lattice: $v_i \in \mathcal{L}$, where $\mathcal{L}$ equals either $\mathcal{U}$ or $\mathcal{H}$ defined below.  Since all paths begin and end on the $x$-axis, they have even length.

We distinguish between two sets of such paths, $\mathcal{D}_{2n}$ and $\mathcal{B}_{2n}$, the sets of  \emph{Dyck paths} and \emph{bilateral Dyck paths} of length $2n$ respectively.   The former set consists of paths confined to the upper half plane: $\mathcal{U}:=\mathbb{Z}_{\ge 0} \times \mathbb{Z}_{\ge 0}$, and the latter superset consists of paths which obey no such restriction, whose vertices may occur anywhere within the half plane $\mathcal{H}:=\mathbb{Z}_{\ge 0} \times \mathbb{Z}$.

\subsection{Bi-banded weighting} \label{subsec:bibanded}

A path is said to be \emph{bi-banded} if it is weighted by edges, where the edges of the path are classified according to whether they fall in an \emph{odd} or \emph{even} band as follows.  
\begin{figure}[htbp]
\begin{center}
\includegraphics[height=20mm]{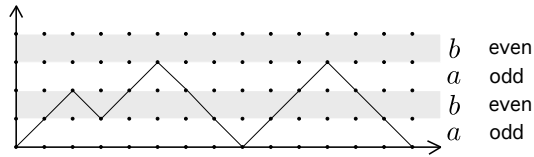}
\caption{A bi-banded Dyck path with weight $a^8b^6$.}
\label{fig:bibandedabOddEven}
\end{center}
\end{figure}
An \emph{up-step} from an even to an odd height is said to lie in an \emph{odd band}.  A \emph{down-step} from an odd to an even height is also said to lie in an \emph{odd band}.  All other edges are said to lie in an \emph{even band}.  The weight of an edge is
\begin{equation} \label{eq:bibanded}
w(e)=
\begin{cases}
a & \text{$e$ is in an odd band}\\
b & \text{$e$ is in an even band}
\end{cases}
\end{equation}
The \emph{bi-banded path weight monomial}, as in the example for a Dyck path in Figure~\ref{fig:bibandedabOddEven}, is defined to be the product of the weights of the edges in the path:
\be \label{eq:monomial}
w(p)=\prod_{e \in p}w(e)
\ee
The \emph{weight polynomial}, or \emph{partition function}, parametrized by path length $t$, is the sum of all the weight polynomials over all bi-banded paths of length $t$: 
\be \label{eq:weightPolynomial}
W_t(a,b)=\sum_{\begin{smallmatrix}
\text{length}(p)=t
\end{smallmatrix}} w(p)
\ee

\subsection{Peak-counting weighting}

A path is said to be \emph{peak-counting} if it is weighted by peaks, where a \emph{peak} is a vertex $v_i$  lying between successive edges $e_{i-1}=(v_{i-1},v_i)$ and $e_i=(v_i,v_{i+1})$, such that $e_{i-1}$ is an \emph{up edge} and $e_i$ is a \emph{down edge}.  
For convenience we define the initial vertex $v_0$ to be a peak if the first edge $(v_0,v_1)$ is a \emph{down-step}; and the final vertex $v_t$ to be a peak if the last edge $(v_{t-1},v_t)$ is an \emph{up-step}.  

It follows from the definition that the first and last vertices of a Dyck path are never peaks.  However for a more general bilateral Dyck path, the first and last vertices may or may not be peaks.  This may be interpreted, if desired, by appending extra edges to the beginning and end of a path.  We imagine all paths to begin at vertex $v_{-1}=(-1,-1)$ with an up-step, and end at vertex $v_{t+1}=(t+1, -1)$ with a down-step, as illustrated in Figure~\ref{fig:peakCountBoxEx}.
\begin{figure}[htbp]
\begin{center}
\includegraphics[height=3cm]{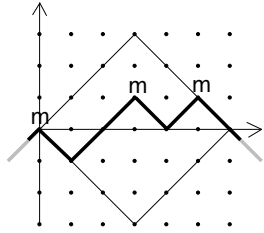}
\caption{A peak-counting bilateral Dyck path with weight $m^3$.}
\label{fig:peakCountBoxEx}
\end{center}
\end{figure}
The weight of a vertex is
\be \label{eq:peaksMonomial}
\widehat{w}(v)=\begin{cases}
m & \text{$v$ is a peak}\\ 
1 & \text{$v$ is \emph{not} a peak}
\end{cases}
\ee
The \emph{peak-counting path weight monomial} is the product of the weights of the vertices:
\be
\widehat{w}(p)=\prod_{v \in p} \widehat{w}(v).
\ee
It is immediate that this is equivalent to 
\be
\widehat{w}(p)=p^{|\text{peaks}|},
\ee
where $|\text{peaks}|$ is the number of peaks.
The peak-counting \emph{weight polynomial}, or \emph{partition function}, parametrized by path length $t$, is the sum of all the weight polynomials over all peak-counting Dyck paths of length $t$: 
\be \label{eq:weightPolynomialcorners}
\widehat{W}_t(m)=\sum_{\begin{smallmatrix}
\text{length}(p)=t
\end{smallmatrix}} \widehat{w}(p)
\ee

\section{Two Path Representations} \label{sec:pathReps}

We use two different representations of paths as input and output for the bijection of the next section.  They are a classical word representation and a new  `checkmark' representation.

\subsection{Word representation} \label{subsec:word}
For a path $p \in \mathcal{B}_{2n}$, the well-known \emph{word representation} consists of a sequence of letters, often $U$ and $D$, which specify the edges from $1$ to $2n$ as up-steps and down-steps.  For later convenience, we use \emph{arrows} to represent  up-steps and \emph{blanks} to represents down-steps.

Conversely,  any sequence comprising exactly $n$ arrows and $n$ blanks generates a unique lattice path which we choose to begin at $v_0=(0,0)$ and of necessity ends at $v_{2n}=(2n,0)$.  
\begin{figure}[htbp]
\begin{center}
\includegraphics[height=25mm]{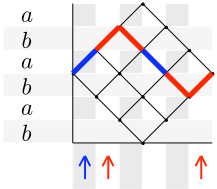}
\caption{A characterization of the path illustrated as the finite sequence: ``arrow, arrow, blank, blank, blank, arrow".}
\label{fig:updownarrowExNonums}
\end{center}
\end{figure}
This is the characterization used in Step~\ref{it:updownarrows} of the algorithm defining the bijection below in Section~\ref{sec:bijection}, and it is clearly reversible.  An example of the characterization is illustrated in Figure~\ref{fig:updownarrowExNonums}.  

\textbf{Note on bi-banded weighting: } If a path characterized by the word representation has a bi-banded weighting, then every arrow with an odd index corresponds to an up-step in an odd band, thus weighted `$a$', and these are the only up-steps in the path weighted `$a$'.  Also, for each up-step weighted `$a$', there is exactly one down-step weighted `$a$'.  Thus the weight of a bi-banded path characterized in this way with $u$ odd-indexed arrows, is $a^{2u}b^{2n-2u}$.  Alternatively we write $a^{2n-2v}b^{2v}$, where $v$ is the number of even-indexed arrows.  In the example given in Figure~\ref{fig:updownarrowExNonums}, the weight is $a^2b^4$.

\subsection{Checkmark representation} \label{subsec:checkmark}

A path of length $2n$ may also be characterized using  `checkmarks' alongside the north-west and south-west sides of the bounding box - the square with vertices $(0,0), (n,n), (2n,0)$ and $(n, -n)$.  
An example is given by Figure~\ref{fig:checkMarksExNamed}.
\begin{figure}[htbp]
\begin{center}
\includegraphics[height=25mm]{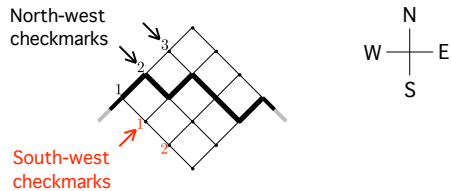}
\caption{A path characterized in terms checkmarks at north-west positions with labels `2' and `3', and south-west position labeled `1'.  The \emph{checkmarks} specify where the path turns.  }
\label{fig:checkMarksExNamed}
\end{center}
\end{figure}

\begin{definition} \label{def:checkMarkRep} Let $n \in \mathbb{N}$.  Then let a \emph{north-west checkmark sequence}, $s_{NW}$, be a sequence of arrows and blanks of length $n$; and a \emph{south-west checkmark sequence}, $s_{SW}$, be a sequence of arrows and blanks of length $n-1$.  If the pair of sequences $(s_{NW}, s_{SW})$ satisfies the following condition
\begin{enumerate}
\item[i.] the number of arrows in the $s_{NW}$ sequence is either equal to or exceeds by one the number of arrows in the $s_{SW}$ sequence 
\end{enumerate}
then it is called  a \emph{checkmark pair}.  
Furthermore, if the pair of checkmark sequences satisfies the additional condition
\begin{enumerate}
\item[ii.] the index of the $h^{\text{th}}$ arrow in $s_{NW}$ (indexed from $1$ to $n$) is strictly greater than the index of the $h^{\text{th}}$ arrow in $s_{SW}$ (indexed from $1$ to $n-1$), for  $h = 1,...,|s_{NW}|_{\uparrow}$,

\end{enumerate}
where $|s_{NW}|_{\uparrow}$ is the number of arrows in $s_{NW}$.  Then 
 the checkmark pair is called a \emph{Dyck checkmark pair}.
\end{definition}

\begin{proposition} \label{thm:checkMarkRep}
Any path $p \in \mathcal{B}_{2n}$ has a unique representation as a checkmark pair.  Conversely, any checkmark pair defines a unique path in $\mathcal{B}_{2n}$. Furthermore, any path $p \in \mathcal{D}_{2n}$ has a unique representation as a Dyck checkmark pair.  Conversely, any Dyck checkmark pair defines a unique path in $\mathcal{D}_{2n}$. 

\end{proposition}

\begin{proof}
Given $p \in \mathcal{B}_{2n}$, north-west and south-west checkmark sequences may be generated as follows.  
\bi
\item Label the first $n$ vertices along the north-west side of the bounding box from `1' to `$n$', starting at the origin.  Place an arrow at each label lying on a NW-to-SE-oriented line in which the path turns right.  The resulting length-$n$ sequence of arrows and blanks is $s_{NW}$.  (Note, any turn on the north-east boundary does not generate an arrow.)

\item Label $n-1$ vertices along the south-west side of the bounding box from $1$ to $n-1$, starting at the $(1,-1)$ vertex and using the first coordinate of the position as the label.  Place an arrow at each label lying on a NE-to-SW-oriented line in which the path turns left.  The resulting length-$(n-1)$ sequence of arrows and blanks is  $s_{SW}$.  (Note, any turn on the south-east boundary does not generate an arrow.)
\ei
This process generates a pair $(s_{NW}, s_{SW})$, whose constituent sequences have the correct lengths.   The number of arrows in $s_{SW}$ is always either equal to or  less by one than the number of arrows in $s_{NW}$.  This is because every time that the path turns down (right) subject to a north-west checkmark, it must turn up (left) again to compensate in order to reach the north-east bounding wall on which it ends at $v_{2n}=(2n,0)$.  However sometimes, as in the example in Figure~\ref{fig:checkMarksExNamed}, the final up-turn (left) occurs on the south-east bounding wall and so is not accompanied by a south-west checkmark.

Conversely, given a checkmark pair,  a unique path may be created by starting at the origin with a north-easterly orientation, continuing in this direction until either a north-west checkmark or the upper bounding wall is found to lie on the same $NW$-to-$SE$-oriented line as the current vertex, turning right, continuing in the new direction until a south-west checkmark or bounding wall is found to lie on the same $NE$-to-$SW$-oriented line as the current vertex, turning left; and then repeating the process until the rightmost vertex of the bounding box is reached.  
All checkmarks must be used in this process, since the number of north-west checkmarks exceeds the number of south-west checkmarks by at most one, and  north-west checkmarks are always used before south-west checkmarks. 

Now, let $p$ be a Dyck path.  Then the vertices of $p$ must stay on or above the baseline $y=0$.  This occurs precisely when the checkmark representation of $p$ satisfies condition (ii) of the Theorem, because every north-west checkmark at position $h$ on the north-west side of the bounding box, turning it downwards, has a corresponding  checkmark at some position strictly less than $h$ on the south-west side of the bounding box, turning the path back upwards before or when it hits the baseline.

Conversely, any pair of checkmark sequences satisfying (ii) generate a path which stays above the baseline, which is therefore a Dyck path.

\end{proof}

\textbf{Note on peak-counting weighting:  } If a path represented by checkmarks is weighted by peaks, then there is a peak for every right turn, as well as a final peak somewhere along the north-east boundary, possibly at the last vertex, as in the example in Figure~\ref{fig:checkMarksExNamed}, which has  weight $m^3$.  Thus the peaking-counting weight of a path with $v$ north-west checkmarks is $m^{v+1}$.

\section{The Bijection} \label{sec:bijection}

The same bijection will suffice for bilateral Dyck paths as for Dyck paths, since the latter are a subset of the former and it turns out that the mapping we describe preserves the property of being a Dyck path.  Hence we begin by describing the bijection for the general class of bilateral Dyck paths.  

\vspace{3mm} 
\begin{theorem} \label{thm:bijection} Let $\phi: \mathcal{B}_{2n} \rightarrow \mathcal{B}_{2n}$ be the function defined by  the following steps.  

\begin{enumerate}

\item \label{it:updownarrows} Express the path $p$ as a sequence, $\{s_j\}_{j=1,..,t}$, of \emph{arrows} and \emph{blanks}, where each arrow stands for an \emph{up-step}, and each blank stands for a \emph{down-step}.

\item \label{it:subseq} 

	\begin{enumerate}
	\item \label{it:split} Extract two subsequences, $s_\text{odd}$ and $s_\text{even}$, comprising the odd and even indexed entries of $s$, respectively.  Then
	
	\item \label{it:toggle} Create sequence $s_{NW}$ by toggling the entries of $s_\text{odd}$, i.e. record a \emph{blank} for each \emph{arrow}, and an \emph{arrow} for each original \emph{blank}.
	
	\item \label{it:removelast} Create sequence $s_{SW}$ by copying $s_\text{even}$ and removing the last entry (which may be either a \emph{blank} or an \emph{arrow}).
	\end{enumerate}

\item \label{it:checkmarks} Reconstitute a new path $\phi(p)$ using the pair $(s_{NW}, s_{SW})$ as north-west and south-west \emph{checkmark} sequences, as described in Section~\ref{subsec:checkmark} above.
\end{enumerate}
Then $\phi$ is a bijection.    Furthermore, if the path $p$ has bi-banded weighting
\be \label{eq:pathweightab}
w(p) = a^{2n-2v}b^{2v}
\ee
then the path $\phi(p)$ has peak-counting weight
\be \label{eq:pathweightpeak}
\widehat{w}(\phi(p)) = m^{v+1}.
\ee

\end{theorem}

\begin{proof}
The word and checkmark representations of paths which constitute Steps~\ref{it:updownarrows} and \ref{it:checkmarks} were already shown to be invertible when they were introduced.  It remains to show Step~\ref{it:subseq} is also invertible.  The inverse of Step~\ref{it:split} is to re-interleave the two subsequences that it originally split.  Step~\ref{it:toggle} toggles arrows with blanks, which process is its own inverse.  Step~\ref{it:removelast} removes the final element of a sequence of blanks and arrows. No information is lost in this step, since the removed element is specified by the constraint that the recreated length-$2n$ sequence should have an equal number of arrows and blanks; thus we add whichever is necessary to complete the sequence.    Hence $\phi$ is a well-defined bijection from bilateral Dyck paths to bilateral Dyck paths.  Furthermore the notes on weights in Subsections~\ref{subsec:word} and \ref{subsec:checkmark} give Equations~(\ref{eq:pathweightab}) and (\ref{eq:pathweightpeak}).
\end{proof}

An example of the bijective mapping $\phi$ is given by Figures~\ref{fig:updownarrowExNonums} and \ref{fig:checkMarksExNamed} as path $p$ and image $\phi(\p)$ respectively.

\begin{corollary} Let $\phi |_{\mathcal{D}_{2n}}: \mathcal{D}_{2n} \rightarrow \mathcal{D}_{2n}$ be the restriction of the bijection $\phi$ defined in Theorem~\ref{thm:bijection} to the class of Dyck paths $\mathcal{D}_{2n} \subset \mathcal{B}_{2n}$.  Then $\phi |_{\mathcal{D}_{2n}}$ is a bijection.  Furthermore, equations~(\ref{eq:pathweightab}) and (\ref{eq:pathweightpeak}) still hold where $p$ is a Dyck path subject to the bi-banded weighting and $\phi|_{\mathcal{D}_{2n}}(p)$ its image subject to the peak-counting weighting.

\end{corollary}

\begin{proof} It is sufficient to show that $\phi$ and $\phi^{-1}$ both map Dyck paths into Dyck paths.  (For a more detailed argument than the one which follows, see \cite{O}.)

Let $p\in\mathcal{D}_{2n}$, where we regard $p$ as having a bi-banded weighting.  Now, $\phi(p)$ is a Dyck path iff it satisfies condition (ii) of Theorem~\ref{def:checkMarkRep}.  This condition must hold in $\phi(p)$, because every north-west checkmark in the peak-counting path corresponds to a downstep that is an odd-indexed edge of the original path, and that downstep in the original path must have been preceded by an upstep in the original path, since the original path was a Dyck path.  The preceding upstep corresponds to the desired south-west checkmark.  

Conversely, let $p\in\mathcal{D}_{2n}$, where we regard $p$ as being weighted by peaks.  Transform $p$ according to the map $\phi^{-1}$, and regard the new path $\phi^{-1}(p)$ as having a bi-banded weighting.   As previously noted, in the checkmark representation of the peak-counting path $p$, every north-west checkmark with a given index has a corresponding south-west checkmark with strictly smaller index. Each north-west checkmark becomes a downstep in the new path $\phi^{-1}(p)$.  The corresponding south-west checkmarks become preceding upsteps, which keep the new path on or above the baseline, as required.   
\end{proof}
\begin{corollary}\label{thm:bilateral Dyck_and_Narayana} Let $W_{2n}(a,b)$ and $\widehat{W}_{2n}(a,b)$ be, respectively, the weight polynomials for bi-banded paths and peak-counting paths, of length $2n>0$ on the lattice $\mathcal{L}$. Then
\be \label{eq:Wbibanded}
W_{2n}(a,b) = \sum_{v=0}^{n-1}c_{n,v}a^{2n-2v}b^{2v}
\ee
and 
\be \label{eq:Wpeak}
\widehat{W}_{2n}(a,b) = \sum_{v=0}^{n-1}c_{n,v}m^{v+1}
\ee
where in the case  $\mathcal{L}$ equal to the half plane $\mathcal{H}$, corresponding to bilateral Dyck paths, the coefficients $c_{n,v} = \binom{n}{v}^2$;
and in the case  $\mathcal{L}$ equal to the upper half plane $\mathcal{U}$, corresponding to Dyck paths, the coefficients  $c_{n,v}=N_{n,v}$, with
\be
N_{n,v}:= \frac{1}{v+1}\binom{n}{v}\binom{n-1}{v} 
\ee 
being the Narayana numbers.  
\end{corollary}

\begin{proof}  Equations~(\ref{eq:Wbibanded}) and (\ref{eq:Wpeak}) follow directly from Theorem~\ref{thm:bijection}.  To calculate the coefficients in the case of bilateral Dyck-paths, it is convenient to use the peak-counting weighting.  Note that a path with $v+1$ peaks has $v$ north-west checkmarks, and either $v$ or $v-1$ south-west checkmarks, all of which may be freely distributed.  Hence there are exactly
\be
c_{n,v}= \binom{n}{v}\left (\binom{n-1}{v} + \binom{n-1}{v-1}\right ) = \binom{n}{v}^2
\ee
peak-counting bilateral Dyck paths with weight $m^{v+1}$.  For Dyck paths the peak-counting result in terms of Narayana numbers \cite{N1955, N1979} is well-known.
\end{proof}

\subsection{Comments}  The bijection $\phi$ was found after many failed attempts to `act geometrically' on Dyck paths by applying rules such as reflections which related to a poset structure.  With the perspective of hindsight, it seems unlikely that these kinds of rules would have produced the required bijection, since we needed to relate lattice paths that were weighted in very different ways.  Bi-banded paths are weighted by edges, and the weighting can be seen as part of the underlying lattice, with edges `picking up' weights as they pass over the bands in the lattice.  By contrast, peak-counting paths are weighted by vertices, and the weights cannot be interpreted as sitting on the underlying lattice - instead the weights are an intrinsic property of the shape of the path.  To find this bijection $\phi$ we widened our perspective to a larger class of paths, the bilateral Dyck paths, and then acted `non-locally' on them.

\section*{Acknowledgements} Financial support from the Australian Research Council via its support for the Centre of Excellence for Mathematics and Statistics of Complex Systems (MASCOS) is gratefully acknowledged by the author.

\end{document}